\documentclass[12pt, reqno]{amsart}
\usepackage{amsmath, amsthm, amscd, amsfonts, amssymb, graphicx, color}
\usepackage[bookmarksnumbered, colorlinks, plainpages]{hyperref}
\hypersetup{colorlinks=true,linkcolor=red, anchorcolor=green, citecolor=cyan, urlcolor=red, filecolor=magenta, pdftoolbar=true}

\newtheorem{theorem}{Theorem}[section]

\newtheorem{corollary}[theorem]{Corollary}
\theoremstyle{definition}

\theoremstyle{remark}
\newtheorem{remark}[theorem]{Remark}
\numberwithin{equation}{section}

\newcommand{\be}{\begin{equation}}
\newcommand{\ee}{\end{equation}}

\newcommand{\cM}{{\mathcal M}}
\newcommand{\NN}{\mathbb{N}}

\begin{document}
\setcounter{page}{1}

\title[Bounds on the joint and generalized spectral radius]{Bounds on the joint and generalized spectral radius of Hadamard geometric mean of bounded sets of positive kernel operators}

\author[A. Peperko]{Aljo\v{s}a Peperko$^{1,2} $}

\address{$^{1}$  Faculty of Mechanical Engineering, University of Ljubljana, A\v{s}ker\v{c}eva 6, SI-1000 Ljubljana, Slovenia;
\newline
$^{2}$ Institute of Mathematics, Physics, and Mechanics,
Jadranska 19, SI-1000 Ljubljana, Slovenia.}
\email{\textcolor[rgb]{0.00,0.00,0.84}{aljosa.peperko@fmf.uni-lj.si; aljosa.peperko@fs.uni-lj.si}}





\begin{abstract}
 Let $\Psi _1, \ldots \Psi _m$ be bounded sets of positive kernel operators on a Banach function space $L$. We prove that for the generalized spectral radius $\rho$ and the joint spectral radius 
$\hat{\rho}$ the inequalities
$$\rho \left(\Psi _1 ^{\left( \frac{1}{m} \right)} \circ \cdots \circ \Psi _m  ^{\left( \frac{1}{m} \right)} \right) \le \rho (\Psi _1 \Psi _2 \cdots \Psi _m) ^{\frac{1}{m}},$$
$$\hat{\rho}  \left(\Psi _1 ^{\left( \frac{1}{m} \right)} \circ \cdots \circ \Psi _m  ^{\left( \frac{1}{m} \right)} \right) \le \hat{\rho}   (\Psi _1 \Psi _2 \cdots \Psi _m) ^{\frac{1}{m}}$$
 hold, where $\Psi _1 ^{\left( \frac{1}{m} \right)} \circ \cdots \circ \Psi _m  ^{\left( \frac{1}{m} \right)}$ denotes the 
 Hadamard (Schur) geometric mean of the sets $\Psi _1, \ldots , \Psi _m$.
\end{abstract} \maketitle

\section{Introduction}

In \cite{Zh09}, X. Zhan conjectured that, for non-negative $n\times n$ matrices $A$ and $B$, the spectral radius $\rho (A\circ B)$ of the Hadamard product satisfies
$$\rho (A\circ B) \le \rho (AB),$$
where $AB$ denotes the usual matrix product of $A$ and $B$. This conjecture was confirmed by K.M.R. Audenaert in \cite{Au10} via a trace description of the spectral radius. Soon after, this inequality was reproved, generalized and refined in different ways by 
several authors (\cite{HZ10},  \cite{Hu11}, \cite{S11}, \cite{Sc11}, \cite{P12}, \cite{CZ15}, \cite{DP16}).
Applying a fact that the Hadamard product is a principal submatrix of the Kronecker product (i.e., by applying the technique used by R.A. Horn and F. Zhang  of \cite{HZ10}), Z. Huang proved that 
\be
\rho (A_1 \circ A_2 \circ \cdots \circ A_m) \le \rho (A_1 A_2 \cdots A_m)
\label{Hu}
\ee
for $n\times n$ non-negative matrices $A_1, A_2, \cdots, A_m$ (see \cite{Hu11}). The author of the current paper extended the inequality (\ref{Hu}) to non-negative matrices that define bounded 
operators on Banach sequence spaces in \cite{P12}. Moreover, in \cite[Theorem 3.16]{P12}  he generalized this inequality to the setting of the generalized and the joint spectral radius of bounded sets of such non-negative matrices. 
In the proofs certain results on the Hadamard product from \cite{DP05} and \cite{P06} were used.

Earlier, A.R. Schep was the first one to observe that the results  \cite{DP05} and \cite{P06} are applicable in this context (see \cite{S11} and \cite{Sc11}). In particular, in \cite[Theorem 2.8]{S11}  he proved that the inequality
\be
\rho \left(A ^{\left( \frac{1}{2} \right)} \circ B  ^{\left( \frac{1}{2} \right)} \right) \le \rho (AB) ^{\frac{1}{2}}
\label{Schep}
\ee
holds for positive kernel operators on $L^p$ spaces. Here $A ^{\left( \frac{1}{2} \right)} \circ B  ^{\left( \frac{1}{2}\right)} $ denotes the Hadamard geometric mean of operators $A$ and $B$. In \cite[Theorem 3.1]{DP16}, R. Drnov\v{s}ek and the author, generalized this inequality 
and proved that the inequality 
\be 
\rho \left(A_1^{\left(\frac{1}{m}\right)} \circ A_2^{\left(\frac{1}{m}\right)} \circ \cdots \circ 
A_m^{\left(\frac{1}{m}\right)}\right)   \le \rho (A_1 A_2 \cdots A_m)^{\frac{1}{m}} 
\label{genHuBfs}
\ee
holds for positive kernel operators $A_1, \ldots, A_m$ on an arbitrary Banach function space.
 
The article is partly expository, since it also includes some new proofs of known results. It is organized as follows. In the second section we introduce some definitions and facts, and we recall some results from \cite{DP05} and \cite{P06}, 
which we will need in our proofs. In the third section we give a new proof of a key result from \cite{DP05} and \cite{P06} (Theorem \ref{ideaOmlDrn}) and recall what this result actually means in the setting of the generalized and the joint spectral radius (Theorem \ref{family}).
In our main result (Theorem \ref{genHuDP}) we generalize the inequality (\ref{genHuBfs}) to the setting of the generalized and the joint spectral radius of bounded sets of positive kernel operators on an arbitrary Banach function space. 


\section{Preliminaries}
\vspace{1mm}

Let $\mu$ be a $\sigma$-finite positive measure on a $\sigma$-algebra $\cM$ of subsets of a non-void set $X$.
Let $M(X,\mu)$ be the vector space of all equivalence classes of (almost everywhere equal)
complex measurable functions on $X$. A Banach space $L \subseteq M(X,\mu)$ is
called a {\it Banach function space} if $f \in L$, $g \in M(X,\mu)$,
and $|g| \le |f|$ imply that $g \in L$ and $\|g\| \le \|f\|$. Throughout the article, it is assumed that  $X$ is the carrier of $L$, that is, there is no subset $Y$ of $X$ of 
 strictly positive measure with the property that $f = 0$ a.e. on $Y$ for all $f \in L$ (see \cite{Za83}).

Standard examples of Banach function spaces are Euclidean spaces,  the space $c_0$ 
of all null convergent sequences  (equipped with the usual norms and the counting measure), the
well-known spaces $L^p (X,\mu)$ ($1\le p \le \infty$) and other less known examples such as Orlicz, Lorentz,  Marcinkiewicz  and more general  rearrangement-invariant spaces (see e.g. \cite{BS88}, \cite{CR07} and the references cited there), which are important e.g. in interpolation theory.
 Recall that the cartesian product $L=E\times F$ 
of Banach function spaces is again a Banach function space, equipped with the norm
$\|(f, g)\|_L=\max \{\|f\|_E, \|g\|_F\}$.

By an {\it operator} on a Banach function space $L$ we always mean a linear
operator on $L$.  An operator $A$ on $L$ is said to be {\it positive} 
if it maps nonnegative functions to nonnegative ones, i.e., $AL_+ \subset L_+$, where $L_+$ denotes the positive cone $L_+ =\{f\in L : f\ge 0 \; \mathrm{a.e.}\}$.
Given operators $A$ and $B$ on $L$, we write $A \ge B$ if the operator $A - B$ is positive.
Recall that a positive  operator $A$ is always bounded, i.e., its operator norm
\be
\|A\|=\sup\{\|Ax\|_L : x\in L, \|x\|_L \le 1\}=\sup\{\|Ax\|_L : x\in L_+, \|x\|_L \le 1\}
\label{equiv_op}
\ee
is finite.  
Also, its spectral radius $\rho (A)$ is always contained in the spectrum.

An operator $A$ on a Banach function space $L$ is called a {\it kernel operator} if
there exists a $\mu \times \mu$-measurable function
$a(x,y)$ on $X \times X$ such that, for all $f \in L$ and for almost all $x \in X$,
$$ \int_X |a(x,y) f(y)| \, d\mu(y) < \infty \ \ \ {\rm and} \ \ 
   (Af)(x) = \int_X a(x,y) f(y) \, d\mu(y)  .$$
One can check that a kernel operator $A$ is positive iff 
its kernel $a$ is non-negative almost everywhere. Observe that (finite or infinite) non-negative matrices that define operators on Banach sequence spaces are a special case of positive kernel operators 
(see e.g. \cite{P12}, \cite{DP16}, \cite{DP10} and the references cited there).  It is well-known that kernel operators play a very important, often even central, role in a variety of applications from differential and integro-differential equations, problems from physics 
(in particular from thermodinamics), engineering, statistical and economic models, etc (see e.g. \cite{J82}, \cite{BP03}, \cite{LL05}, \cite{DLR13} 
and the references cited there).
For the theory of Banach function spaces and more general Banach lattices we refer the reader to the books \cite{Za83}, \cite{BS88}, \cite{AA02}, \cite{AB85}. 

Let $A$ and $B$ be positive kernel operators on $L$ with kernels $a$ and $b$ respectively,
and $\alpha \ge 0$.
The \textit{Hadamard (or Schur) product} $A \circ B$ of $A$ and $B$ is the kernel operator
with kernel equal to $a(x,y)b(x,y)$ at point $(x,y) \in X \times X$ which can be defined (in general) 
only on some order ideal of $L$. Similarly, the \textit{Hadamard (or Schur) power} 
$A^{(\alpha)}$ of $A$ is the kernel operator with kernel equal to $(a(x, y))^{\alpha}$ 
at point $(x,y) \in X \times X$ which can be defined only on some order ideal of $L$.

Let $A_1 ,\ldots, A_n$ be positive kernel operators on a Banach function space $L$, 
and $\alpha _1, \ldots, \alpha _n$ positive numbers such that $\sum_{j=1}^n \alpha _j = 1$.
Then the {\it  Hadamard weighted geometric mean} 
$A = A_1 ^{( \alpha _1)} \circ A_2 ^{(\alpha _2)} \circ \cdots \circ A_n ^{(\alpha _n)}$ of 
the operators $A_1 ,\ldots, A_n$ is a positive kernel operator defined 
on the whole space $L$, since $A \le \alpha _1 A_1 + \alpha _2 A_2 + \ldots + \alpha _n A_n$ by the inequality between the weighted arithmetic and geometric means. Let us recall  the following result which was proved in \cite[Theorem 2.2]{DP05} and 
\cite[Theorem 5.1]{P06}. 

\begin{theorem} 
Let $\{A_{i j}\}_{i=1, j=1}^{k, m}$ be positive kernel operators on a Banach function space $L$.
If $\alpha _1$, $\alpha _2$,..., $\alpha _m$ are positive numbers  
such that $\sum_{j=1}^m \alpha _j = 1$, then the positive kernel operator
$$A:= \left(A_{1 1}^{(\alpha _1)} \circ \cdots \circ A_{1 m}^{(\alpha _m)}\right) \ldots \left(A_{k 1}^{(\alpha _1)} \circ \cdots \circ A_{k m}^{(\alpha _m)} \right)$$
satisfies the following inequalities
\begin{eqnarray}
\label{basic2}
A &\le &  
(A_{1 1} \cdots  A_{k 1})^{(\alpha _1)} \circ \cdots 
\circ (A_{1 m} \cdots A_{k m})^{(\alpha _m)} , \\
\label{norm2}
\left\|A \right\| &\le &  
\|A_{1 1} \cdots  A_{k 1}\|^{\alpha _1} \cdots \|A_{1 m} \cdots A_{k m}\|^{\alpha _m}, \\ 
\label{spectral2}
\rho \left(A \right) &\le  &
\rho \left( A_{1 1} \cdots  A_{k 1} \right)^{\alpha _1} \cdots 
\rho \left( A_{1 m} \cdots A_{k m}\right)^{\alpha _m} .
\end{eqnarray}

\label{DPBfs}
\end{theorem}
The following result is a special case  of Theorem \ref{DPBfs}. 
\begin{theorem} 
\label{special_case}
Let $A_1 ,\ldots, A_m$ be positive kernel operators on a Banach function space  $L$,
and $\alpha _1, \ldots, \alpha _m$ positive numbers such that $\sum_{j=1}^m \alpha _j = 1$.
Then we have
\be
 \|A_1 ^{( \alpha _1)} \circ A_2 ^{(\alpha _2)} \circ \cdots \circ A_m ^{(\alpha _m)} \| \le
  \|A_1\|^{ \alpha _1}  \|A_2\|^{\alpha _2} \cdots \|A_m\|^{\alpha _m}  
\label{gl1nrm}
\ee
and
\be
 \rho(A_1 ^{( \alpha _1)} \circ A_2 ^{(\alpha _2)} \circ \cdots \circ A_m ^{(\alpha _m)} ) \le
\rho(A_1)^{ \alpha _1} \, \rho(A_2)^{\alpha _2} \cdots \rho(A_m)^{\alpha _m} .
\label{gl1vecr}
\ee

\end{theorem}

Recall also that the above results on the spectral radius and operator norm  remain valid under less restrictive assumption $\sum_{j=1}^m \alpha _j \ge 1$ in the case of  (finite or infinite) non-negative matrices that define operators on sequence spaces 
(\cite{EJS88}, \cite{DP05}, \cite{P06}, \cite{P12}, \cite{DP16}).

Let $\Sigma$ be a bounded set of bounded operators on $L$.
For $m \ge 1$, let 
$$\Sigma ^m =\{A_1A_2 \cdots A_m : A_i \in \Sigma\}.$$
The generalized spectral radius of $\Sigma$ is defined by
\be
\rho (\Sigma)= \limsup _{m \to \infty} \;[\sup _{A \in \Sigma ^m} \rho (A)]^{1/m}
\label{genrho}
\ee
and is equal to 
$$\rho (\Sigma)= \sup _{m \in \NN} \;[\sup _{A \in \Sigma ^m} \rho (A)]^{1/m}.$$
The joint spectral radius of $\Sigma$ is defined by
\be
\hat{\rho}  (\Sigma)= \lim _{m \to \infty}[\sup _{A \in \Sigma ^m} \|A\|]^{1/m}.
\label{BW}
\ee
It is well known that $\rho (\Sigma)= \hat{\rho}  (\Sigma)$ for a precompact set $\Sigma$ of compact operators on $L$ (see e.g. \cite{ShT00}, \cite{ShT08}, \cite{Mo}), 
in particular for a bounded set of complex $n\times n$ matrices (see e.g. \cite{BW92}, \cite{E95}, \cite{SWP97}, \cite{Dai11}, \cite{MP12}).
This equality is called the Berger-Wang formula or also the 
generalized spectral radius theorem (for an elegant proof in the finite dimensional case see \cite{Dai11}).
However, in general $\rho (\Sigma)$ and $\hat{\rho}  (\Sigma)$ may differ even in the case of a bounded set $\Sigma$ of compact positive operators on $L$ as the 
following example from \cite{SWP97} shows. Let
$\Sigma =\{A_1, A_2, \ldots \}$ be a bounded set of compact operators on $L=l^2$ defined by $A_k e_k =e_{k+1}, (k \in \NN) $ and 
$A_k e_j =0$ for $j\neq k$.
Then $(A_{i_1}A_{i_2}\cdots A_{i_k})^2=0$ for arbitrary $k \in \NN$ and any subset $\{i_1, i_2, \ldots, i_k\} \subset \NN$.
Thus $\rho(\Sigma)=0$. Since 
$$A_mA_{m-1} \cdots A_1 e_1 =e_{m+1}, \;\;\; m \in \NN , $$
$$A_mA_{m-1} \cdots A_1 e_j =0, \;\;\; j\neq 1 ,$$
we have
$\hat{\rho}(\Sigma)\ge \limsup _{m \to \infty} \|A_m \cdots A_1\| ^{1/m}=1$
and so $\rho(\Sigma) \neq \hat{\rho}(\Sigma)$.

In \cite{Gui82}, the reader can find an example of two positive non-compact weighted shifts $A$ and $B$ on $L=l^2$ such that $\rho(\{A,B\})=0 < \hat{\rho}(\{A,B\})$.

The theory of the generalized and the joint spectral radius has many important applications for instance to discrete and differential inclusions, 
wavelets, invariant subspace theory
(see e.g. \cite{BW92}, \cite{Dai11}, \cite{Wi02}, \cite{ShT00}, \cite{ShT08} and the references cited there).
In particular, $\hat{\rho} (\Sigma)$ plays a central role in determining stability in convergence properties of discrete and differential inclusions. In this 
theory the quantity $\log \hat{\rho} (\Sigma)$ is known as the maximal Lyapunov exponent (see e.g. \cite{Wi02}).

We will  use the following well known facts that
$$\rho (\Sigma  ^m) = \rho (\Sigma)^m ,\;\; \hat{\rho} (\Sigma  ^m) = \hat{\rho} (\Sigma)^m ,\;\;   \rho (\Psi \Sigma) = \rho (\Sigma\Psi) \;\;\mathrm{and}\;\; \hat{\rho} (\Psi \Sigma) = \hat{\rho} (\Sigma\Psi),$$
where $\Psi \Sigma =\{AB: A\in \Psi, B\in \Sigma\}$ and $m\in \NN$.

\section{Results}

First we provide a new proof of the inequality (\ref{spectral2}) (based on its special case (\ref{gl1vecr})) by applying the method of  proof of the inequality (\ref{genHuBfs}) from \cite[Theorem 3.1]{DP16}. 
\begin{theorem} Let $\{A_{i j}\}_{i=1, j=1}^{k, m}$ be positive kernel operators on a Banach function space $L$
and let $\alpha _1$, $\alpha _2$,..., $\alpha _m$  be positive numbers such that $\sum_{j=1}^m \alpha _j = 1$.  Then the inequality (\ref{spectral2}) holds.
\label{ideaOmlDrn}
\end{theorem}
\begin{proof}
If $A_1, \ldots, A_k$ are positive kernel operators on $L$, then 
the block matrix 
$$ T = T(A_1, A_2, \ldots, A_k) := \left[
\begin{matrix}  
0 & A_1 & 0 &  0 & \ldots & 0 & 0 \cr
0 & 0 &  A_2 & 0  & \ldots & 0 & 0 \cr
0 & 0 & 0 & A_3  & \ldots & 0 & 0 \cr
\vdots & \vdots & \vdots & \ddots & \ddots & \vdots &\vdots \cr
\vdots & \vdots & \vdots & \vdots & \ddots & \ddots &\vdots \cr
0 & 0 & 0 & 0 & \ldots & 0 & A_{k-1} \cr 
A_k & 0 & 0 & 0 & \ldots & 0 & 0
\end{matrix}  \right].  $$
defines a positive kernel operator on the cartesian product of $k$ copies of $L$. Since $T^k$ has a block diagonal form
$$ T^k = \textrm{diag} \, \left( A_1 A_2 \cdots A_k,  A_2 A_3 \cdots A_k A_1,   A_3 A_4 \cdots A_k A_1 A_2, 
\ldots,  A_k A_1 A_2 \cdots A_{k-1} \right) , $$
we have  $\rho (T)^k = \rho (T^k) = \rho ( A_1 A_2 \cdots A_k )$.

Now define $T_i := T(A_{1i},A_{2i}, \ldots, A_{ki})$ for $i =1, 2, \ldots, m$. 
Then $\rho (T_i) = \rho ( A_{1i} A_{2i} \cdots A_{ki} )^{1/k}$ for each  $i =1, 2, \ldots, m$.  Using the inequality (\ref{gl1vecr}) we obtain that
$$  \rho \left(T_1^{\left(\alpha _1\right)} \circ T_2^{\left(\alpha _2\right)} \circ \cdots \circ 
T_m^{\left(\alpha _m \right)}\right) \le \ \rho(T_1) ^{\alpha _1} \, \rho(T_2)^{\alpha _2}  \cdots \rho(T_m)^{\alpha _m}  $$
$$= \left(\rho \left( A_{1 1} \cdots  A_{k 1} \right)^{\alpha _1} \cdots 
\rho \left( A_{1 m} \cdots A_{k m}\right)^{\alpha _m}\right)^{1/k},  $$
since $\sum_{j=1}^m \alpha _j = 1$.
On the other hand, 
$$T_1^{\left(\alpha _1\right)} \circ T_2^{\left(\alpha _2\right)} \circ \cdots \circ T_m^{\left(\alpha _m \right)} = T (C_1, \ldots, C_k)$$
where
$$C_j = \left(A_{j 1}^{(\alpha _1)} \circ \cdots \circ A_{j m}^{(\alpha _m)}\right) $$
for all $j=1, 2, \ldots , k$. It follows that  $\rho (T_1^{\left(\alpha _1\right)} \circ T_2^{\left(\alpha _2\right)} \circ \cdots \circ T_m^{\left(\alpha _m \right)} )=\rho (A)^{1/k}$, where
$$A= \left(A_{1 1}^{(\alpha _1)} \circ \cdots \circ A_{1 m}^{(\alpha _m)}\right) \cdots 
\left(A_{k 1}^{(\alpha _1)} \circ \cdots \circ A_{k m}^{(\alpha _m)}\right),$$
which proves the inequality  (\ref{spectral2}). 
\end{proof}
\begin{remark} In the case of non-negative matrices that define operators  on a Banach  sequence space (see e.g. \cite{P12}, \cite{DP16}, \cite{DP05}, \cite{P06} for exact definitions),  the same proof works for 
positive numbers $\alpha _1$, $\alpha _2$,..., $\alpha _m$  such that $\sum_{j=1}^m \alpha _j \ge 1$.
\end{remark}

Let $\Psi _1, \ldots , \Psi _m$ be bounded sets of positive kernel operators on a Banach function space $L$ and let $\alpha _1, \ldots \alpha _m$ be positive numbers such that 
$\sum _{i=1} ^m \alpha _i = 1$. Then the bounded set of positive kernel operators on $L$, defined by
$$\Psi _1 ^{( \alpha _1)} \circ \cdots \circ \Psi _m ^{(\alpha _m)}=\{ A_1 ^{( \alpha _1)} \circ \cdots \circ A _m ^{(\alpha _m)}: A_1\in \Psi _1, \ldots, A_m \in \Psi _m \},$$
is called the {\it weighted Hadamard (Schur) geometric mean} of sets $\Psi _1, \ldots , \Psi _m$. The set  
$\Psi _1 ^{(\frac{1}{m})} \circ \cdots \circ \Psi _m ^{(\frac{1}{m})}$ is called the  {\it Hadamard (Schur) geometric mean} of sets $\Psi _1, \ldots , \Psi _m$.

A version of the following result on the generalized and the joint spectral radius was stated in \cite[Theorem 3.4]{P12} and \cite[Corollary 5.3]{P06} only in the case of bounded sets  of  
non-negative matrices that define operators on Banach sequence spaces, 
however the same proof works in our more general setting by applying the inequalities (\ref{spectral2}) and (\ref{norm2}). The proof is included for the convenience of the reader.
\begin{theorem}
Let $\Psi _1, \ldots \Psi _m$ be bounded sets of positive kernel operators on a Banach function space $L$ and let 
 $\alpha _1, \ldots \alpha _m$ be positive numbers such that \\
$\sum _{i=1} ^m \alpha _i = 1$.
Then we have 
\be
\rho (\Psi _1 ^{( \alpha _1)} \circ \cdots \circ \Psi _m ^{(\alpha _m)} ) \le 
\rho (\Psi _1)^{ \alpha _1} \, \cdots \rho(\Psi _m)^{\alpha _m} 
\label{gsh}
\ee
and
\be
\hat{\rho } (\Psi _1 ^{( \alpha _1)} \circ \cdots \circ \Psi _m ^{(\alpha _m)} ) \le 
\hat{\rho }  (\Psi _1)^{ \alpha _1} \, \cdots \hat{\rho } (\Psi _m)^{\alpha _m}. 
\label{gsh2}
\ee
\label{family}
\end{theorem}
\begin{proof}
Let $A \in (\Psi _1 ^{( \alpha _1)} \circ \cdots \circ \Psi _m ^{(\alpha _m)})^l$, $l \in \NN$. Then there are 
$A_{ik} \in \Psi _k$, $i=1, \ldots ,l$, $k=1,\ldots , m$ such that
$$A=(A_{11} ^{\alpha _1} \circ \cdots \circ A_{1m} ^{\alpha _m}) \cdots (A_{l1} ^{\alpha _1} \circ \cdots \circ A_{lm} ^{\alpha _m}).$$
By Theorem \ref{DPBfs} we have 
\be
\rho(A) \le \rho (A_{11} \cdots A_{l1})^{\alpha _1} \cdots \rho (A_{1m} \cdots A_{lm})^{\alpha _m}.
\label{useg}
\ee
Since $A_{1k} \cdots A_{lk} \in \Psi_k ^l$ for all $k=1,\ldots , m$, (\ref{useg}) implies (\ref{gsh}).

By replacing $\rho (\cdot)$ with $\|\cdot\|$ in the proof above, we obtain the inequality (\ref{gsh2}), which completes the proof.

\end{proof}

Now we  prove our main result, which is  a generalization of the inequality (\ref{genHuBfs}).  It can also be considered as a kernel version of \cite[Theorem 3.16]{P12}, which holds for bounded sets of non-negative matrices that define operators on Banach sequence spaces.
\begin{theorem} Let $\Psi _1, \ldots \Psi _m$ be bounded sets of positive kernel operators on a Banach function space $L$.
Then we have
\be
\rho \left(\Psi _1 ^{\left( \frac{1}{m} \right)} \circ \cdots \circ \Psi _m  ^{\left( \frac{1}{m} \right)} \right) \le \rho (\Psi _1 \Psi _2 \cdots \Psi _m) ^{\frac{1}{m}}
\label{Hurho}
\ee
and
\be
\hat{\rho}  \left(\Psi _1 ^{\left( \frac{1}{m} \right)} \circ \cdots \circ \Psi _m  ^{\left( \frac{1}{m} \right)} \right) \le \hat{\rho}   (\Psi _1 \Psi _2 \cdots \Psi _m) ^{\frac{1}{m}}.
\label{Hurhoh}
\ee
\label{genHuDP}
\end{theorem}
\begin{proof}
To prove (\ref{Hurho}) we will show that 
$$\rho \left(\Psi _1 ^{\left( \frac{1}{m} \right)} \circ \cdots \circ \Psi _m  ^{\left( \frac{1}{m} \right)} \right)^m \le \rho (\Psi _1 \Psi _2 \cdots \Psi _m). $$
Take $A\in  \left(\Psi _1 ^{\left( \frac{1}{m} \right)} \circ \cdots \circ \Psi _m  ^{\left( \frac{1}{m} \right)} \right)^{mk}$. Then $A=A_1 A_2\cdots A_k$, where
$$A_i= \left(A_{i \,1 \,1} ^{\left( \frac{1}{m} \right)} \circ A_{i\, 1\, 2} ^{\left( \frac{1}{m} \right)} \circ \cdots \circ A_{i\, 1\, m} ^{\left( \frac{1}{m} \right)}\right)\left(A_{i\, 2\, 1} ^{\left( \frac{1}{m} \right)}\circ A_{i\, 2\, 2} ^{\left( \frac{1}{m} \right)} \circ \cdots \circ A_{i\, 2\, m} ^{\left( \frac{1}{m} \right)}\right) \cdots $$
$$\cdots \left(A_{i m 1} ^{\left( \frac{1}{m} \right)}\circ A_{i m 2} ^{\left( \frac{1}{m} \right)} \circ \cdots \circ A_{i m m} ^{\left( \frac{1}{m} \right)}\right) $$
for some $A_{i\, j\, 1} \in \Psi _1, \ldots, A_{i\, j \,m} \in \Psi _m$ and all $j=1, \ldots, m$, $i=1, \ldots, k$. Then  
$$A_i= \left(A_{i\, 1\, 1}  ^{\left( \frac{1}{m} \right)} \circ A_{i\, 1\, 2}  ^{\left( \frac{1}{m} \right)} \circ \cdots \circ A_{i\, 1\, m}  ^{\left( \frac{1}{m} \right)}\right)\left(A_{i\, 2\, 2}  ^{\left( \frac{1}{m} \right)} \circ \cdots \circ A_{i \,2 \,m} ^{\left( \frac{1}{m} \right)}\circ A_{i\, 2\, 1}  ^{\left( \frac{1}{m} \right)}\right) \cdots$$
$$  \cdots \left(A_{i\, m \,m}  ^{\left( \frac{1}{m} \right)}\circ A_{i\, m\, 1}  ^{\left( \frac{1}{m} \right)} \circ \cdots \circ A_{i\, m\, m-1} ^{\left( \frac{1}{m} \right)} \right). $$
By (\ref{basic2}) we have
$$A=A_1 A_2\cdots A_k \le B_1 ^{\left( \frac{1}{m} \right)}  \circ B_2 ^{\left( \frac{1}{m} \right)} \circ \cdots \circ B_m ^{\left( \frac{1}{m} \right)},$$
where
$$B_1 = \prod _{i=1} ^k A_{i\, 1\, 1}A_{i\, 2\, 2} \cdots A_{i\, m \,m} \in (\Psi _1 \Psi _2 \cdots \Psi _m)^k,$$
$$B_2 = \prod _{i=1} ^k A_{i\, 1\, 2}A_{i\, 2\, 3} \cdots A_{i\, m\, 1} \in (\Psi _2 \Psi _3 \cdots \Psi _1)^k,$$
$$\cdots \cdots \cdots$$
$$B_m = \prod _{i=1} ^k A_{i\, 1\, m}A_{i\, 2\, 1} \cdots A_{i\, m\, m-1} \in (\Psi _m \Psi _1 \cdots \Psi _{m-1})^k.$$
By Theorem \ref{special_case} we have
$$\rho (A) \le  \rho (B_1) ^{\frac{1}{m}} \rho (B_2)^{\frac{1}{m}} \cdots \rho (B_m)^{\frac{1}{m}},$$
which implies 
$$\rho  \left(\Psi _1 ^{\left( \frac{1}{m} \right)} \circ \cdots \circ \Psi _m  ^{\left( \frac{1}{m} \right)} \right)^m \le \left(\rho (\Psi _1 \Psi _2 \cdots \Psi _m)\rho (\Psi _2 \Psi _3 \cdots \Psi _1)\cdots \rho (\Psi _m \Psi _1 \cdots \Psi _{m-1}) \right) ^{\frac{1}{m}} $$
$$= \rho (\Psi _1 \Psi _2 \cdots \Psi _m).$$
This proves (\ref{Hurho}) and the inequality (\ref{Hurhoh}) is proved similarly.
\end{proof}
The following special case of Theorem \ref{genHuDP} generalizes (\ref{Schep}).
\begin{corollary} Let $\Psi $ and $\Sigma$  be bounded sets of positive kernel operators on a Banach function space $L$.
Then we have
\be
\rho \left(\Psi ^{\left( \frac{1}{2} \right)} \circ \Sigma  ^{\left( \frac{1}{2} \right)} \right) \le \rho (\Psi \Sigma) ^{\frac{1}{2}}
\label{Hurho2}
\ee
and
\be
\hat{\rho}  \left(\Psi ^{\left( \frac{1}{2} \right)} \circ \Sigma  ^{\left( \frac{1}{2} \right)} \right) \le \hat{\rho} (\Psi \Sigma) ^{\frac{1}{2}}
\label{Hurhoh2}
\ee
\label{genHu}
\end{corollary}

\noindent {\bf Acknowledgements.} The author thanks Professor Franz Lehner for reading the first version of this article.

This work was supported in part by the JESH grant of the Austrian Academy of Sciences and by grant P1-0222 of the Slovenian Research Agency. 

\bibliographystyle{amsplain}

\end{document}